\documentclass[12pt]{article}

\usepackage{amsmath,amsthm}
\usepackage{amssymb}

\usepackage{enumerate}

\usepackage{graphicx}

\newcommand{\Far}{\text{Far}}
\newcommand{\co}{\text{co}}

\newtheorem{theorem}{Theorem}[section]
\newtheorem{corollary}[theorem]{Corollary}

\newtheorem{proposition}[theorem]{Proposition}

\theoremstyle{definition}
\newtheorem{definition}[theorem]{Definition}
\newtheorem{remark}[theorem]{Remark}
\newtheorem{example}[theorem]{Example}

\numberwithin{equation}{section}

\begin{document}


\baselineskip=17pt


\title{On a generalisation of Local Uniform Rotundity}

\author{Uday Shankar Chakraborty\\
Department of Mathematics\\ 
Assam University, Silchar\\
Assam, India\\
E-mail: udayhkd@gmail.com}

\date{}

\maketitle


\renewcommand{\thefootnote}{}



\renewcommand{\thefootnote}{\arabic{footnote}}
\setcounter{footnote}{0}


\begin{abstract}
In this paper we investigate the property (HLUR), a generalisation of (LUR) property of a Banach space. A Banach space having the property (HLUR) is called an HLUR space. We characterise (HLUR) property with the help of known geometric properties and study various properties of HLUR spaces. We show that for any finite dimensional Banach space, the property (HLUR) coincides with anti-Daugavet property of the space. We also show some applications of HLUR spaces in connection with farthest points of sets.
\end{abstract}

\section{Introduction}
Let $X$ be a real Banach space and $X^*$ its dual. For a sequence $(x_n)\subset X$ and for $x\in X$, the norm and weak convergence of $(x_n)$ to $x$ are denoted by $x_n\to x$ and $x_n\xrightarrow{w}x$ respectively. $B_X$ and $S_X$ denote the closed unit ball and unit sphere in $X$ respectively. For $0\leq \delta<1, x\in S_X$ and $x^*\in S_{X^*}$, let $S(X,x^*,\delta)=\{x\in B_X: x^*(x)\geq 1-\delta\}$ and $S(X^*,x,\delta)=\{x^*\in B_{X^*}: x^*(x)\geq 1-\delta\}$ be the slices of $B_X$ and $B_{X^*}$ determined by $x^*, \delta$ and $x,\delta$ respectively. For $x\in S_X, x^*\in S_{X^{*}}$, let $J(x)=\{y^*\in S_{X^*}: y^*(x)=1\}$ and  $A_0(x)=\bigcup\{S(X,x^*,0):x^*\in J(x)\}$. For $x\in S_X$ and $0\leq \delta\leq 1$, let $$D[x,\delta]=\{y\in B_X: \|\frac{x+y}{2}\|\geq 1-\delta\} \text{ and }C[x,\delta]=D[x,\delta]\cap S_X.$$ It is clear that $C[x,0]=D[x,0]=A_0(x)=\{y\in S_X: \|x+y\|=2\}$. \\
\begin{definition}
$X$ is said to be locally uniformly rotund (LUR)\cite{Lovagila} if for all $x\in S_X, x_n\in S_X (n\in \mathbb{N})$ with $\|x_n+x\|\to 2$, we have $x_n\to x$.
\end{definition}

 Local uniform rotundity plays an important role in the study of geometry of Banach spaces. Since its introduction, Banach spaces with that property have been thoroughly studied over the years from various perspectives. Several generalisations of locally uniformly rotund spaces have been suggested by many mathematicians and their properties have also been extensively studied. One such natural generalisation is the property (HLUR), introduced in \cite{Chakrabarty}. 
 
 \begin{definition}
$X$ is said to have the property (HLUR) if for every $x, x_n\in S_X(n\in \mathbb{N})$ with $\|x_n+x\|\to 2$ and for every $x^*\in J(x)$, we have $d(x_n,S(X,x^*,0))\to 0$. 
 \end{definition}
It is easy to verify that if $X$ satisfies the property (HLUR), then $d(x_n,S(X,x^*,0))\to 0$ whenever $x\in S_X$, $x_n\in B_X$ for every $n\in \mathbb{N}$ and $x^*\in J(x)$.\\\\
 We call a space satisfying the property (HLUR), an HLUR space. Every LUR space is HLUR and the converse is true for rotund spaces \cite{Chakrabarty}.\\\\ Although in recent times numerous generalisations of LUR spaces have been studied by many authors across the globe, as for instance see \cite{Banas, Cabrera, Dutta1, Dutta2, Revalski}, it seems that the notion (HLUR) remained unnoticed. Motivated by this, in this present study, we explore various aspects of HLUR spaces including its geometric properties and certain applications in connection with farthest distance and bounded linear operators satisfying the Daugavet equation.

\section{Preliminaries}
In this section we enlist some important definitions which will be used in subsequent sections.

\begin{definition} 
\begin{enumerate}[(i)]
\item $X$ is said to satisfy the property (HS)\cite{Chakrabarty} if for every $x^*\in S_{X^*}$ with $S(X,x^*,0)\neq \emptyset$ and for every sequence $(x_n)$ in $B_X$ with $x^*(x_n)\to 1$, we have $d(x_n, S(X,x^*,0))\to 0$.

\item $X$ is said to be compactly locally uniformly rotund (CLUR)\cite{Vlasov} , if for every $x\in S_X$ and $x_n\in S_X (n\in \mathbb{N})$ with $\|x_n+x\|\to 2$, $(x_n)$ has a convergent subsequence. 

\item $X$ is said to be weakly compactly locally uniformly rotund (weakly CLUR)\cite{Dutta1} , if for every $x\in S_X$ and $x_n\in S_X (n\in \mathbb{N})$ with $\|x_n+x\|\to 2$, $(x_n)$ has a subsequence that converges weakly to some point in $A_0(x)$. 

\item $X$ is said to be rotund if the cardinality of $S(X,x^*,0)$ is at most one for all $x^*\in S_{X^*}$.

\item $X$ is called smooth if for each $x\in S_X$, $J(x)$ is a singleton.

\item $X$ is said to satisfy the Kadets-Klee property (KK)\cite{Meggin} if the norm and weak convergence in $S_X$ coincide. Equivalently, $X$ is said to satisfy (KK) if $x_n \to x$ whenever $(x_n)$ is a sequence in $X$ and $x\in X$ such that $x\xrightarrow{w}x, \|x_n\|\to \|x\|$.

\item $X$ is said to be strongly rotund\cite{Meggin} if it is reflexive, rotund and it satisfies the property (KK).

\item $X$ is said to be alternatively convex or smooth (ACS)\cite{Kadets} if for all $x,y\in S_X$ and for all $x^*\in S_{X^*}$, $x^*(y)=1$ whenever $\|x+y\|=2$ and $x^*(x)=1$.

\item $X$ is said to have property (I) \cite{Edelstein} if every closed and bounded convex set in $X$ can be expressed as the intersection of a family of closed balls.
\end{enumerate}
\end{definition}

We know that $X$ is called a $U$-convex space \cite{Lau} if for every $\epsilon>0$ there exists $\delta>0$ such that for all $x,y\in S_X$ with $\|\frac{x+y}{2}\|\geq 1-\delta$ implies $x^*(y)>1-\varepsilon$ for all $x^*\in J(x)$. In \cite{Dutta2}, two localisations of $U$-convexity were suggested. 

\begin{definition}$X$ is called 
\begin{enumerate}[(i)]
\item locally $U$-convex \cite{Dutta1} if for each $x,x_n(n\in \mathbb{N})\in S_X$ with $\|x_n+x\|\to 2$, there exists $x^*\in J(x)$ such that $x^*(x_n)\to 1$.

\item strongly locally $U$-convex if for each $x,x_n(n\in \mathbb{N})\in S_X$ with $\|x_n+x\|\to 2$ and for each $x^*\in J(x)$, $x^*(x_n)\to 1$.
\end{enumerate}

\end{definition} 

\begin{definition} Let $K$ be a closed convex subset of $X$.
\begin{enumerate}[(i)]
\item A subset $C$ of $K$ is said to be an exposed face of $K$ if there exists $x^*\in X^*$ such that $C=\{x\in K: x^*(x)=\sup x^*(K)\}$.

\item A subset $C$ of $K$ is said to be a strongly exposed face of $K$ if there exists $x^*\in X^*$ satisfying $C=\{x\in K: x^*(x)=\sup x^*(K)\}$ and for every open subset $U$ of $K$ containing $C$, there exists $\delta>0$ such that $\{x\in K: x^*(x)\geq \sup x^*(K)-\delta\}$ is contained in $U$.

\item A point $x\in K$ is said to be an exposed (res. strongly exposed) point of $K$ if $\{x\}$ is an exposed (res. strongly exposed) face of $K$.
\end{enumerate}
\end{definition}

\begin{definition}
A sequence of sets $(C_n)$ in $X$ is said to converge to $C_0\subset X$ in upper Hausdorff sense, denoted by $C_n\xrightarrow{H^{+}}C_0$, if for every $\varepsilon>0$, $C_n\subset C_0+\varepsilon B_X$ eventually, and in lower Hausdorff sense denoted by $C_n\xrightarrow{H^{-}}C_0$, if for every $\varepsilon>0$, $C_0\subset C_n+\varepsilon B_X$ eventually. A sequence $(C_n)$ converges to $C_0$ in Hausdorff sense denoted by $C_n\xrightarrow{H}C_0$, if $C_n\xrightarrow{H^{+}}C_0$ and $C_n\xrightarrow{H^{-}}C_0$. 
\end{definition}
For a nested sequence of subsets $(C_n)$ of $X$, i.e. for a sequence of subsets $(C_n)$ of $X$ with $C_{n+1}\subset C_n$ for all $n\in \mathbb{N}$, $C_n\xrightarrow{H}C_0 $ if and only if $C_n\xrightarrow{H^{+}}C_0 $.

\begin{definition}
$X$ is said to be anti-Daugavet \cite{Kadets} if for every continuous linear operator $T:X\to X$ satisfying the Daugavet equation $$\|I+T\|=1+\|T\|,$$ where $I:X\to X$ is the identity operator, $\|T\|$ is an approximate point spectrum of $T$.
\end{definition}

\begin{definition}
A Banach space $X$ is said to be asymptotically isometric to $\ell^1$ \cite{Dowling} if it has a normalised Schauder basis $(x_n)$ such that for a sequence $(\lambda_n)$ in $(0,1)$ increasing to $1$, we have that $$\sum\limits_{n=1}^{\infty}\lambda_n |t_n|\leq \|\sum\limits_{n=1}^{\infty} t_n x_n\|$$ for all $(t_n)\in \ell^1$. \\\\ If $(x_n)$ is a normalised sequence in $X$ satisfying the above inequality, then the closed linear span of $(x_n)$ is an asymptotically isometric copy of $\ell^1$.
\end{definition}

\section{Banach spaces with the property (HS)}

\begin{theorem}
$X$ satisfies the property (HS) if and only if every exposed face of $B_X$ is strongly exposed.
\end{theorem}

\begin{proof}
Let $X$ have the property (HS) and let $C$ be an exposed face of $B_X$. Therefore there exists $x^*\in S_{X^*}$ such that $C=S(X,x^*,0)$. If possible, let $S(X,x^*,0)$ be not a strongly exposed face of $B_X$. Then there exists an open subset $U$ of $B_X$ containing $S(X,x^*,0)$ such that for every $\delta >0$, $S(X,x^*,\delta)\not\subset U$. It follows that for each $n\in \mathbb{N}$, there exists $x_n\in S(X,x^*,\frac{1}{n})\setminus U$. Therefore $x^*(x_n)\to 1$ but $d(x_n,S(X,x^*,0))\not\to 0$. This is contrary to the fact that $X$ satisfies the property (HS). Hence $C$ must be a strongly exposed face of $B_X$.\\ Conversely, let every exposed face of $B_X$ be strongly exposed. To show that $X$ satisfies the property (HS), suppose $x^*\in S_{X^*}$ such that $S(X,x^*,0)\neq \emptyset$ and let $(x_n)$ be a sequence in $B_X$ such that $x^*(x_n)\to 1$. Thus $S(X,x^*,0)$ is an exposed face of $B_X$ and so by hypothesis, is a strongly exposed face. Let $U$ be any open subset of $B_X$ containing $S(X,x^*,0)$. Therefore, there exits $\delta>0$ such that $S(X,x^*,\delta)\subset U$. Now $x^*(x_n)\to 1$ implies that $x_n\in S(X,x^*,\delta)$ eventually, and so $x_n\in U$ eventually. Hence $d(x_n, S(X,x^*,0))\to 0$, entailing $X$ to satisfy the property (HS).
\end{proof}

\begin{corollary}
Let $X$ satisfy the property (HS). Then every exposed point of $B_X$ is strongly exposed.
\end{corollary}

\begin{definition}\cite{Cabrera} 
$X$ is said to be nearly strictly convex (NSC) if for every $x^*\in S_{X^*}$, $S(X,x^*,0)$ is compact.
\end{definition}
The following result is not new and can be proved with the help of \cite[Theorem 19]{Shun}. However, we provide a straightforward proof for the sake of completeness.
\begin{theorem}
Let $X$ be a  reflexive space. Then the following statements are equivalent.
\begin{enumerate}[(a)]
\item $X$ satisfies the property (KK).

\item $X$ is NSC and has the property (HS).
\end{enumerate}
 \end{theorem}

\begin{proof}
(a)$\implies$(b): Since $B_X$ is weakly compact, therefore for every $x^*\in S_{X^*}$, $S(X,x^*,0)$ is weakly compact and so is compact by the property (KK) satisfied by $X$, entailing $X$ to be NSC. Again, let $(x_n)\subset B_X$ and $x^*\in S_{X^*}$ be such that $x^*(x_n)\to 1$. Let $(x_{n_k})$ be a subsequence of $(x_n)$ and let $(x_{n_{k_j}})$ be a subsequence of $(x_{n_k})$ converging weakly to $x$ in $B_X$. Since $$x^*(x_{n_{k_j}})\to x^*(x),$$ $x^*(x)=1$. Therefore $\|x\|=1$ and by weak lower semicontinuity of the norm, $\|x_{n_{k_j}}\|\to 1$. Since $X$ has the property (KK), $x_{n_{k_j}}\to x$. Thus $$d(x_{n_{k_j}}, S(X,x^*,0))\to 0.$$ Hence $X$ satisfies the property (HS).\\ 
(b)$\implies$(a): Let $(x_n)\subset S_X$ be such that $x_n\xrightarrow{w}x$ in $S_X$. Suppose $x^*\in J(x)$. Then for every subsequence $(x_{n_k})$ of $(x_n)$, $x^*(x_{n_k})\to 1$ and so $$d(x_{n_k},S(X,x^*,0))\to 0.$$ It follows that there is a sequence $(y_{n_k})$ in $S(X,x^*,0)$ such that $\|x_{n_k}-y_{n_k}\|\to 0$. But $S(X,x^*,0)$ being compact, there exists a subsequence $(y_{n_{k_j}})$ of $(y_{n_k})$ that converges to some $y$ in $S(X,x^*,0)$. Clearly $x_{n_{k_j}}\to y$. Therefore, $(x_n)$ converges to $y$. Since $x$ is the weak limit of $(x_n)$, we have $y=x$. Hence $X$ has the property (KK).
\end{proof}

The following is an easy consequence of the above result.

\begin{corollary}
Every strongly convex space satisfies the property (HS).
\end{corollary}

\section{Characterisations and properties of HLUR spaces}

 \begin{theorem}
Let $X$ be an HLUR space. Then for every $x\in S_X$, $S(X,x^*,0)$ and $S(X,y^*,0)$ coincide for all $x^*,y^*\in J(x)$.
\end{theorem}

\begin{proof}
Let $x\in S_X$ and $x^*\in J(x)$. Suppose that $y\in A_0(x)$. Then $\|x+y\|=2$ and so $d(y,S(X,x^*,0))=0$. Since $S(X,x^*,0)$ is closed in $X$, we have $y\in S(X,x^*,0)$. Hence the result.
\end{proof}

\begin{remark}
From the above result, it is clear that in an HLUR space, every $x\in S_X$ is contained in exactly one exposed face of $B_X$.
\end{remark}

\begin{theorem}
Let $X$ be a Banach space. Then the following are equivalent.

\begin{enumerate}[(a)]
\item $X$ is an HLUR space.

\item For each $x\in S_X$, $D[x,\frac{1}{n}]\xrightarrow{H}A_0(x)$ and $S(X,x^*,0)$ is same for all $x^*\in J(x)$.

\item For each $x\in S_X$, $C[x,\frac{1}{n}]\xrightarrow{H}A_0(x)$ and $S(X,x^*,0)$ is same for all $x^*\in J(x)$.

\item $X$ is strongly locally $U$-convex and satisfies the property (HS). 

\item $X$ is locally $U$-convex, $X$ satisfies the property (HS) and $S(X,x^*,0)$ is same for all $x^*\in J(x)$, for all $x\in S_X$.

\end{enumerate}
\end{theorem}

\begin{proof}
(a)$\implies$(b): Suppose $x\in S_X$ and $x_n\in D[x,\frac{1}{n}]$ for all $n\in \mathbb{N}$. Then $$\|x_n+x\|\to 2.$$ By (HLUR) property of $X$, for every $x^*\in J(x)$, $$d(x_n,S(X,x^*,0))\to 0.$$ Since $S(X,x^*,0)\subset A_0(x)$ for all $x^*\in J(x)$, therefore $$d(x_n,A_0(x))\to 0.$$ Thus for every $\varepsilon >0$, $x_n\in A_0(x)+\varepsilon B_X$ eventually. This follows that $D[x,\frac{1}{n}]\subset A_0(x)+\varepsilon B_X$ eventually. Hence $D[x,\frac{1}{n}]\xrightarrow{H} A_0(x)$. Also by Theorem 4.1, $S(X,x^*,0)$ is same for all $x^*\in J(x)$. \\
(b)$\implies$(c): It is clear from the fact that $C[x,\frac{1}{n}]\subset D[x,\frac{1}{n}]$ for all $x\in S_X$ and for all $n\in \mathbb{N}$.\\
(c)$\implies$(d): Let $x,x_n(n\in \mathbb{N})\in S_X$ be such that $$\|x_n+x\|\to 2.$$ Thus $x_n\in C[x,\frac{1}{N}]$ for every $N\in \mathbb{N}$ eventually. Since $C[x,\frac{1}{n}]\xrightarrow{H}A_0(x)$, thus for every $\varepsilon>0$, $C[x,\frac{1}{n}]\subset A_0(x)+\varepsilon B_X$ eventually, entailing $d(x_n,A_0)\to 0$. Suppose $x^*\in J(x)$. Since $S(X,x^*,0)$ is same for all $x^*\in J(x)$, $A_0(x)=S(X,x^*,0)$, and consequently $$d(x_n,S(X,x^*,0))\to 0.$$ Thus $X$ is HLUR and so, satisfies the property (HS). Again there is a sequence $(y_n)$ in $S(X,x^*,0)$ such that $\|x_n-y_n\|\to 0$. Therefore $$|x^*(x_n)-1|\leq \|x_n-y_n\|\to .0$$ It follows that $x^*(x_n)\to 1$ and hence $X$ is strongly locally $U$-convex. \\
(d)$\implies$(a): To show that $X$ is HLUR, let us consider $x,x_n\in S_X(n\in \mathbb{N})$ such that $$\|x_n+x\|\to 2.$$  If $x^*\in J(x)$, then by the strong local $U$-convexity of $X$, $x^*(x_{n})\to 1$. Since $X$ has the property (HS), therefore $$d(x_n,S(X,x^*,0))\to 0.$$ Hence $X$ is an HLUR space. \\
(a)$\implies$(e): It is enough to show that $X$ is locally $U$-convex. Consider $x,x_n(n\in \mathbb{N})\in S_X$ such that $$\|x_n+x\|\to 2.$$ Since $X$ is HLUR, for every $x^*\in J(x)$, $$d(x_n,S(X,x^*,0))\to 0.$$ Thus there exists a sequence $(y_n)$ in $S(X,x^*,0)$ such that $\|x_n-y_n\|\to 0$. It follows that $$|x^*(x_n)-1|=|x^*(x_n-y_n)|\leq \|x_n-y_n\|\to 0,$$ entailing $x^*(x_n)\to 1$. Hence the claim. \\
(e)$\implies $(a):  Let $x,x_n\in S_X (n\in \mathbb{N})$ be such that $$\|x_n+x\|\to 2.$$ Suppose $x^*\in J(x)$. By local $U$-convexity of $X$, there exists $y^*\in J(x)$ such that $y^*(x_n)\to 1$. Also $x^*,y^*\in J(x)$ implies $S(X,x^*,0)=S(X,y^*,0)$. Since $X$ satisfies the  property (HS), therefore $$d(x_n, S(X,x^*,0))\to 0.$$ Hence $X$ is an HLUR space.
\end{proof}
The next result is an immediate consequence of Theorem 4.3.
\begin{corollary}
The following statements are equivalent.

\begin{enumerate}[(a)]
\item $X$ is LUR.

\item For every $x\in S_X$, $D[x,\frac{1}{n}]\xrightarrow{H}A_0(x)$.

\item For every $x\in S_X$, $C[x,\frac{1}{n}]\xrightarrow{H}A_0(x)$.

\item $X$ is rotund, Strongly locally $U$-convex and satisfies the property (HS).

\item $X$ is rotund, locally $U$-convex and has the property (HS).
\end{enumerate}
\end{corollary}
\begin{remark}
It is easy to see that for a Banach space $X$ with dimension less than or equal to 2, $X$ is HLUR implies that $X$ is rotund or smooth. In fact, the converse is also true. Thus we have the following result.
\end{remark}

\begin{theorem}
Let $\dim X\leq 2$. Then $X$ is HLUR if and only if $X$ is rotund or smooth.
\end{theorem}
Kadets\cite{Kadets} proved that in the class of finite dimensional Banach spaces the class of ACS spaces coincides with the class of anti-Daugavet properties. We prove that finite dimensional HLUR spaces also do the same.
\begin{theorem}
Let $X$ be a finite dimensional space. Then $X$ is HLUR if and only if $X$ is an ACS space.
\end{theorem}

\begin{proof}
Let be HLUR. Suppose $x,y\in S_X$ and $x^*\in J(x)$ be such that $\|x+y\|=2$. Then $$d(y,S(X,x^*,0))=0.$$ Since $S(X,x^*,0)$ is a closed subset of $X$, therefore $y\in S(X,x^*,0)$, entailing $x^*(y)=1$. Hence $X$ is an ACS space.\\ Conversely, let $X$ be an ACS space. Let $x,x_n(n\in \mathbb{N})\in S_X$ be such that $\|x_n+x\|\to 2$. Suppose $x^*\in J(x)$. Let $(x_{n_k})$ be any subsequence of $(x_n)$. By compactness of $S_X$, we get a subsequence $(x_{n_{k_j}})$ of $(x_{n_k})$ such that $x_{n_{k_{j}}}\to x_0$. It follows that $\|x_0+x\|=2$ and so $x^*(x_0)=1$. It follows that $$d(x_{n_{k_j}}, S(X,x^*,0))\to 0.$$ Hence $X$ is HLUR.
\end{proof}

\begin{corollary}
Let $X$ be a finite dimensional space. Then $X$ is HLUR if and only if $X$ is anti-Daugavet.
\end{corollary}
For any nonempty bounded subset $K$ of $X$, and for any $x\in X$, the farthest distance of $x$ from $K$ is given by $$\sup\{\|x-k\|:k\in K\}$$ and the set of all such points of $K$ where farthest distance of $x$ is attained, is given by $$F_K(x)=\{k\in K: \|x-k\|=\sup\{\|x-k\|:k\in K\}\}.$$ The collection of all points in $K$ where the farthest distance of some point of $X$ is attained, denoted by $\Far K$, is given by $$\Far K=\bigcup_{x\in X}F_K(x).$$

The first known work in the field of farthest points of sets are due to Asplund \cite{Asplund} and Edelstein \cite{Edelstein}. Asplund \cite{Asplund} proved that if $X$ is a reflexive, LUR space, then the set of points which admit farthest points in a closed and bounded subset of $X$ is dense in $X$. Eldestein \cite{Edelstein} got the similar conclusion for a uniformly rotund space. It is farther proved \cite{Edelstein} that in a uniformly rotund space $X$ staisfying the property (I), the closed convex hull of any closed and bounded subset $K$ of $X$ equals the closed convex hull of $\Far K$. We prove similar results for a reflexive, HLUR space below. 
\begin{theorem}
Let $K$ be a bounded, weakly sequentially closed subset of a reflexive HLUR space $X$. Then 
\begin{enumerate}[(a)]
\item $A=\{x\in X: F_K(x)\neq \emptyset\}$ is dense in $X$ and $X\setminus A$ is of first category in $X$.

\item $\overline{\co}(K)=\overline{\co}(\Far K)$ if $X$ satisfies the property (I).
\end{enumerate}

\end{theorem}

\begin{proof}
\begin{enumerate}[(a)]
\item Following the technique as employed by Asplund \cite{Asplund} (with minor modifications), the result can be proved.

\item Repeating the same set of arguments as mentioned in \cite[Theorem 2]{Edelstein}, the result can be proved by using Theorem 4.9 (a).
\end{enumerate}
\end{proof}

\section{Relationships among various spaces}

\begin{theorem}
Let $X$ be a reflexive space such that both $X$ and $X^*$ are smooth and satisfy the property (HS). Then both the spaces $X$ and $X^*$ are HLUR spaces.
\end{theorem}

\begin{proof}
It is enough to show that $X$ is HLUR. Let $x,x_n(n\in \mathbb{N})\in S_X$ be such that $\|x_n+x\|\to 2$ and let $x^*\in J(x)$. We choose a sequence $(x_n^*)$ in $S_{X^*}$ such that $x_n^*(x_n+x)=\|x_n+x\|$ for all $n\in \mathbb{N}$. Thus $x_n^*(x_n)\to 1$ and $x_n^*(x)\to 1$. It follows that $Q(x)(x_n^*)\to 1$, where $Q$ is the natural embedding of $X$ onto $X^{**}$. By reflexivity of $X^*$, $S(X^{*},Q(x),0)\neq \emptyset$. Since $X^*$ satisfies the property (HS), therefore $d(x_n^*,S(X^{*},Q(x),0))\to 0$. But due to smoothness of $X$, $S(X^*,Q(x),0)=J(x)$ is a singleton, say $\{x^*\}$. Thus $x_n^*\to x^*$, entailing $x^*(x_n)\to 1$. Therefore, by (HS) property of the space $X$, we have $d(x_n,S(X,x^*,0))\to 0$. Hence $X$ is an HLUR space.
\end{proof}

\begin{corollary}
Let $X$ and $X^*$ be strongly rotund space. Then $X$ and $X^*$ both are LUR spaces.
\end{corollary}

\begin{proof}
Here $X$ is reflexive. Thus both the spaces $X$ and $X^*$ are smooth. Also by Corollary 3.4, $X$ and $X^*$ satisfy the property (HS). Since $X, X^*$ are rotund, the result is attained by using Theorem 5.1.
\end{proof}
It is known that if $X^{**}$ is LUR for any equivalent norm, then $X$ is reflexive. Dutta and Lin \cite[Proposition 2.10]{Dutta2} proved that LUR property can also be replaced by strong local $U$-convexity of $X^{**}$ under any equivalent norm. Since every HLUR space is strongly locally $U$-convex, as shown in Theorem 4.1, we have the following result.
\begin{theorem}
Let $X^{**}$ be HLUR under an equivalent norm. Then $X$ is reflexive.
\end{theorem}
The following are sufficient conditions for a space to satisfy the property (HLUR).
\begin{theorem}
Consider the following statements.
\begin{enumerate}[(a)]
\item $X$ is smooth and CLUR.

\item $X$ is ACS and CLUR.

\item $X$ is HLUR.

\end{enumerate}
Then (a)$\implies$(c) and (b)$\implies$(c).
\end{theorem}

\begin{proof}
(a)$\implies$(c): Let $x,x_n\in S_X(n\in \mathbb{N})$, such that $\|x_n+x\|\to 2$ and let $x^*\in J(x)$. Let $(x_{n_k})$ be any subsequence of $(x_n)$. Then $\|x_{n_k}+x\|\to 2$. Since $X$ is a CLUR space, there exists a subsequence $(x_{n_{k_j}})$ of $(x_{n_k})$ such that $x_{n_{k_j}}\to x_0\in A_0(x)$. But as $J(x)=\{x^*\}$, so $A_0(x)=S(X,x^*,0)$. Therefore $$d(x_{n_{k_j}},S(X,x^*,0))=\|x_{n_{k_j}}-x_0\|\to 0.$$ Consequently, $d(x_{n},S(X,x^*,0))=\|x_{n}-x_0\|\to 0$. Hence $X$ is HLUR.  \\\\
(b)$\implies$(c): Let $x,x_n\in S_X(n\in \mathbb{N})$, such that $\|x_n+x\|\to 2$ and let $x^*\in J(x)$. Let $(x_{n_k})$ be any subsequence of $(x_n)$. Then $\|x_{n_k}+x\|\to 2$. Since $X$ is a CLUR space, there exists a subsequence $(x_{n_{k_j}})$ of $(x_{n_k})$ such that $x_{n_{k_j}}\to x_0\in A_0(x)$. Observe that $x^*(x_0)=1$ as $X$ is ACS and so $d(x_{n_{k_j}}, S(X,x^*,0))\to 0$. Hence the result follows.
\end{proof}

However, the condition smoothness or ACS can not dropped from Theorem 5.4 as shown by the following example.

\begin{example}
Consider the space $(\mathbb{R}^2,\|.\|_{\infty})$, where $\|(x,y)\|_{\infty}=\max\{|x|,|y|\}$. Because of finite dimensionality, the space $(\mathbb{R}^2,\|.\|_{\infty})$ is CLUR. But it is not HLUR. In fact take $x=(1,1), x^*=(\frac{1}{2},\frac{1}{2})$ and $x_n=(0,1)$ for all $n\in \mathbb{N}$. Then $\|x_n+x\|_{\infty}=2$ for all $n\in \mathbb{N}$ and $x^*(x)=1$. But $d(x_n,S(X,x^*,0))=1 \text{ for each }n\in \mathbb{N}$.
\end{example}
Quotient space of an HLUR space need not be HLUR, as shown by the following example.
\begin{example}
Define a norm on $\ell^1$ by $$\|x\|=(\|x\|_1^2+\|x\|_2^2)^{\frac{1}{2}},$$ where $\|.\|_p$ denotes the canonical norm on $\ell^p$. Then $\|.\|$ is equivalent to $\|.\|_1$ as $$\|x\|_1\leq \|x\|\leq \sqrt{2}\|x\|_1$$ for all $x\in \ell^1$ and with this norm $\|.\|$, $\ell^1$ is an LUR space \cite[Lemma 13.26]{Fabian}. Consequently $X=(\ell^1,\|.\|)$ is an HLUR space. Consider the basic sequence $(e_n)$ in $X$, where $$e_n=(0,\ldots,\underbrace{1}_{n\text{ th place }},0,\ldots) \text{ for all }n\in \mathbb{N},$$ and the block basic sequence $(y_n)$ of $(e_n)$, where $$y_n=\sum\limits_{k=\frac{1}{2}n(n-1)+1}^{\frac{1}{2}n(n+1)}e_k \text{ for all }n\in \mathbb{N}.$$ Then it is not difficult to see that $$\sum\limits_{n=1}^{\infty}\lambda_n|t_n|\leq \|\sum\limits_{n=1}^{\infty}t_n\frac{y_n}{\|y_n\|}\|,$$ where $\lambda_n=\frac{1}{\sqrt{1+\frac{1}{n}}}\to 1$. Thus $C$, the closed linear span of $(y_n)$ is an asymptotically isometric copy of $\ell^1$. It follows that $(C,\|.\|)$ is linearly isometric to $(\ell^1,\|.\|_0)$, where $\|.\|_0$ is an asymptotically isometrically equivalent norm determined by the sequence $(\lambda_n)$. Let $\{x_n:n\in \mathbb{N}\}$ be a dense subset of the closed unit ball of $(\ell^1,\|.\|_1)$. Define $T:(\ell^1,\|.\|_0)\to (\ell^1,\|.\|_1)$ by $$T((t_n))=\sum\limits_{n=1}^{\infty}\lambda_nt_n x_n \text{ for all }(t_n)\in \ell^1.$$ Following arguments in \cite[Theorem 5.9]{Fabian2}, we can show that the quotient space of $ (\ell^1,\|.\|_0) $ by $\ker T$ is linearly isometric to $(\ell^1,\|.\|_1)$. Since the unit vector basis $(e_n)$ is symmetric in $X$, $C$ being a closed linear span of a block basic sequence of $(e_n)$, is norm-1 complemented in $X$. If $P:X\to C$ is the corresponding projection, then $X/\ker P$ is linearly isometric to $C$. Hence $(\ell^1,\|.\|_1)$ is linearly isometric to a quotient of $X$. But $(\ell^1,\|.\|_1)$ is not an HLUR space. Consider $e_n (n>1),e_1\in S_{\ell^1}$ such that $\|e_1+e_n\|=2$. Choose $e_1^*\in \ell^{\infty}$ such that $e_1^{*}(e_1)=1$. But $\|e_n-e_1\|=2$ for all $n>1$.
 \end{example}
\begin{definition}
A nonempty subset $K$ of $X$ is said to be proximinal if for every $x\in X$, the set $$P_K(x)=\{y\in K:\|x-y\|=d(x,K)\}$$ is nonempty.
\end{definition}
We know that a subspace $Y$ of $X$ is proximinal if and only if $Q(B_X)=B_{X/Y}$, where $Q$ is the canonical quotient map from $X$ onto $X/Y$. The following is a sufficient condition for (HLUR) property of a quotient space.
\begin{theorem}
	Let $X$ be an HLUR space and let $Y$ be a proximinal subspace of $X$. Then $X/Y$ is HLUR.
\end{theorem}

\begin{proof}
	Let $x+Y, x_n+Y\in S_{X/Y}$ for all $n\in \mathbb{N}$ be such that $$\|(x_n+x)+Y\|\to 2.$$ Since $Y$ is proximinal, therefore $Q(B_X)=B_{X/Y}$. Thus $x+Y,x_n+Y\in Q(B_X)$ and so $x,x_n\in B_X$. Again $\|x+Y\|\leq \|x\|, \|x_n+Y\|\leq \|x_n\|$. It follows that $x,x_n\in S_X$. Also $$2\leftarrow\|(x_n+x)+Y\|\leq \|x_n+x\|\leq \|x\|+\|x_n\|\leq 2.$$ Therefore $\|x_n+x\|\to 2$. Suppose $x^*\in J(x+Y)$. Since $(X/Y)^*$ is isometrically isomorphic to $Y^{\perp}=\{y^*\in X^*: y^*(y)=0 \text{ for all }y\in Y\}$, therefore $x^*\in J(x)$. By HLUR property of $X$, $$d(x_n,S(X,x^*,0))\to 0.$$ It follows that there exists a sequence $(x_n^{\prime})$ in  $S(X,x^*,0)$ such that $$\|x_n-x_n^{\prime}\|\to 0.$$ Clearly, $x_n^{\prime}+Y\in S(X/Y,x^*,Y)$ for all $n\in \mathbb{N}$ and $$\|(x_n+Y)-(x_n^{\prime}+Y)\|\leq \|x_n-x_n^{\prime}\|\to 0.$$ It follows that $$d(x_n+Y,S(X/Y,x^*,Y))\to 0.$$ Hence $X/Y$ is an HLUR space.
\end{proof}

As it is mentioned in \cite[p.77]{Chakrabarty}, HLUR spaces are not comparable with CLUR spaces, it becomes worthwhile to ask under what condition an HLUR space becomes CLUR. We provide such a condition below. 
\begin{theorem}
If $X$ is an HLUR space which is also NSC, then $X$ is CLUR.
\end{theorem}

\begin{proof}
Suppose $x,x_n\in S_X(n\in \mathbb{N})$ such that $\|x_n+x\|\to 2$. Then for all $x^*\in J(x)$, $d(x_n,S(X,x^*,0))\to 0$. Thus there exists a sequence $(y_n)$ in $S(X,x^*,0)$ such that $\|x_n-y_n\|\to 0$. Since $X$ is an NSC space, $S(X,x^*,0)$ is compact. Therefore there is a subsequence $(y_{n_k})$ of $(y_n)$ such that $(y_{n_k})$ converges to some $y$ in $S(X,x^*,0)$. Thus $$\|x_{n_k}-y\|\leq \|x_{n_k}-y_{n_k}\|+\|y_{n_k}-y\|\to 0.$$ It follows that $x_{n_k}\to y$ and so $X$ is CLUR.
\end{proof}

\begin{definition}\cite{Cabrera} 
$X$ is said to be weakly nearly strictly convex (WNSC) if for every $x^*\in S_{X^*}$, $S(X,x^*,0)$ is weakly compact.
\end{definition}
Clearly, every reflexive space is necessarily a WNSC space.
\begin{theorem}
Every HLUR space which is also WNSC is weakly CLUR.
\end{theorem}
\begin{proof}
Let $X$ be a WNSC, HLUR space. Suppose $x,x_n(n\in \mathbb{N})$ such that $\|x_n+x\|\to 2$ and suppose $x^*\in J(x)$. Then by HLUR property of $X$, $d(x_n,S(X,x^*,0))\to 0$. Thus there exists a sequence $(y_n)$ in $S(X,x^*,0)$ such that $\|x_n-y_n\|\to 0$. Since $X$ is a WNSC space, $S(X,x^*,0)$ is weakly compact and so there exists a subsequence $(y_{n_k})$ of $(y_n)$ converges weakly to some $y$ in $S(X,x^*,0)\subset A_0(x)$. Now for any $y^*\in X^*$, \begin{align*}
& |y^*(x_{n_k})-y^*(y)|\\\leq& |y^*({x_{n_k}})-y^*(y_{n_k})|+|y^*(y_{n_k})-y^*(y)|\\ \leq &\|y^*\|\|x_{n_k}-y_{n_k}\|+|y^*(y_{n_k})-y^*(y)|\to 0.
\end{align*} 
Hence $X$ is weakly CLUR.
\end{proof}
Chakrabarty \cite{Chakrabarty} proved that every CLUR space satisfies the  property (HS). We present the proof here for completeness. 

\begin{proposition}\cite[Proposition 4.3.1]{Chakrabarty}
Every CLUR space satisfies the (HS) property.
\end{proposition}

\begin{proof}
Let $X$ be a CLUR space. Let $x^*\in S_{X^*}$ be such that $S(X,x^*,0)\neq \emptyset$. Let $(x_n)$ be a sequence in $B_X$ such that $x^*(x_n)\to 1$. If $x\in S(X,x^*,0)$, then for any subsequence $(x_{n_k})$ of $(x_n)$, we have $$2\geq \|x_{n_k}+x\|\geq x^*(x_{n_k}+x)=x^*(x_{n_k})+x^*(x)\to 2,$$ which implies that $\|x_{n_k}+x\|\to 2$. Since $X$ is CLUR, $(x_{n_k})$ has a subsequence $(x_{n_{k_j}})$ converging to some point $x_0\in X$. Clearly $x_0\in S(X,x^*,0)$, and so $d(x_{n_{k_j}}, S(X,x^*,0))\to 0$. Consequently $d(x_{n}, S(X,x^*,0))\to 0$. Hence $X$ has the property (HS).
\end{proof}
The following result characterises the property (CLUR) in terms of nearly strict convexity, local $U$-convexity and the property (HS).
\begin{theorem}
The following statements are equivalent.

\begin{enumerate}[(a)]
\item $X$ is CLUR.

\item $X$ is weakly CLUR, NSC and $X$ satisfies the property (HS).

\item $X$ is locally $U$-convex, NSC and $X$ satisfies the property (HS).
\end{enumerate}
\end{theorem}

\begin{proof}
(a)$\implies$(b): By Proposition 5.12, $X$ satisfies the property (HS). Also by \cite[Corollary 3.5]{Dutta1}, $X$ is weakly CLUR.
Since $X$ is CLUR, therefore $A_0(x)$ is compact. Thus for each $x^*\in S_{X^*}$, $S(X,x^*,0)$ being a closed set is compact. Thus $X$ is NSC.\\\\
(b)$\implies$(c): It follows from the fact that every weakly CLUR space is locally $U$-convex.\\\\
(c)$\implies$(a): To show that $X$ is CLUR, let $x,x_n(n\in \mathbb{N})\in S_X$ be such that $\|x_n+x\|\to 2$. Since $X$ is locally $U$-convex, therefore there exists $x^*\in J(x)$ such that $x^*(x_n)\to 1$. Thus by the property (HS) of $X$, we have $$d(x_n,S(X,x^*,0))\to 0.$$ It follows that there exists a sequence $(y_n)$ in $S(X,x^*,0)$ such that $\|x_n-y_n\|\to 0$. But $S(X,x^*,0)$ is compact and so there exists a subsequence $(y_{n_k})$ of $(y_n)$ that converges to some $y_0\in S(X,x^*,0)$. It follows that $(x_{n_k})$ converges to $y_0$. Hence $X$ is CLUR.
\end{proof}
\textbf{Acknowledgements}: The author would like to thank the anonymous referee for valuable comments which helped to improve the quality of the paper. Thanks are due to Prof. Vladimir Kadets, 
and Prof. William B. Johnson, for their valuable suggestions regarding the Example 5.6 through Researchgate and mathoverflow platforms respectively. The author is also grateful to Dr. Anjan Kumar Chakrabarty for his continued support.

\end{document}